\documentclass[reqno]{amsart}
\usepackage{amsmath,amssymb,amsthm, epsfig}
\usepackage{xcolor}

\usepackage[usenames,dvipsnames]{pstricks}
\usepackage{epsfig}
\usepackage{pst-grad}
\usepackage{pst-plot}
\usepackage{xcolor}
\definecolor{azul}{RGB}{33,64,154}

\usepackage{hyperref}
\hypersetup{
colorlinks,
urlcolor=azul,
linkcolor=azul,
citecolor=azul,
pdftitle={Singular infinity Laplacian equations}}

\def \R {\mathbb{R}}
\def \dist {\mathrm{dist}}
\newtheorem{theorem}{Theorem}

\newtheorem{proposition}{Proposition}
\newtheorem{corollary}{Corollary}

\newtheorem{remark}{Remark}

\numberwithin{equation}{section}

\title[Infinity Laplacian equations with singular absorptions]{Infinity Laplacian equations with singular absorptions}

\author[D.J. Ara\'ujo]{Dami\~ao J. Ara\'ujo}
\address{UFPB, Department of Mathematics, Universidade Federal da Para\'iba, 58059-900, Jo\~ao Pessoa-PB, Brazil}{}
\email{araujo@mat.ufpb.br}

\author[G.S. S\'a]{Ginaldo S. S\'a}
\address{UFPB, Department of Mathematics, Universidade Federal da Para\'iba, 58059-900, Jo\~ao Pessoa-PB, Brazil}{}
\email{ginaldo.sa@academico.ufpb.br}

\begin{document}

\begin{abstract}
In this work, we study regularity properties for nonvariational singular elliptic equations ruled by the infinity Laplacian. We obtain optimal $C^{1,\alpha}$ regularity along the free boundary. We also show existence of solutions, nondegeneracy properties and fine geometric estimates for the free boundary.
\end{abstract}

\date{\today}

\maketitle

\tableofcontents

\section{Introduction}
In this work, we study analytic and geometric properties of nonvariational elliptic equations with singular absorption terms, where the governing second-order operator is the infinity Laplace operator
$$
\Delta_\infty u := \sum\limits_{ij}D_{ij}uD_{i}uD_{j}u.
$$
This highly degenerate operator has received wide attention during the last three decades. This operator is strongly related to models which describe, for example, random tug-of-war games \cite{PSSW1} and mass transfer problems \cite{EG}. Infinity harmonic functions, {\it i.e.}, solutions of $\Delta_\infty u =0$, correspond to the best Lipschitz extension problem and the notion of comparison with cones, see \cite{A1,A2,ACJ,CEG}.

The lack of uniform ellipticity makes the mathematical study of models related to the infinity Laplacian more delicate. Existence and uniqueness results are well established, however the regularity of infinity harmonic functions remains one of the most challenging issues in the modern theory of nonlinear pdes. It is well known that infinity harmonic functions are locally Lipschitz, and the best regularity result to date was established by Evans and Smart, whom have proved differentiability everywhere \cite{ESm}. In two dimensions, Evans and Savin obtained $C^{1,\beta}$ regularity \cite{ES} for some $\beta$ universally small, see also \cite{S1}. The infinity harmonic function
$$
x^{4/3}-y^{4/3}, \quad (x,y) \in \mathbb{R}^2
$$
suggests the optimal regularity is H\"older continuity of the first order derivatives with exponent $1/3$. We also mention \cite{KZZ1,KZZ2}, where important results concerning planar sharp Sobolev regularity are obtained.
For the inhomogeneous infinity Laplace equation,
$$
\Delta_\infty u = f(x,u) \in L^\infty,
$$
existence and uniqueness of viscosity solutions of the Dirichlet problem have been established \cite{LW}, where, for a bounded source term $f$, solutions are Lipschitz continuous. In the case $f \in C^1$, everywhere differentiability has been established \cite{L}. Nonetheless, to the best of our knowledge, no further regularity is known.

Free boundary problems involving the infinity Laplacian have also been investigated. For the infinity-obstacle problem \cite{RTU}, solutions grow at the sharp rate $4/3$ near the contact set. In \cite{ALT}, the authors considered absorption terms $f(x,u) = u_+^\theta$, for $0 \leq \theta < 3$, obtaining, in particular, that nonnegative solutions are smoother along the boundary of the noncoincidence set $\partial\{u>0\}$, see also \cite{DT}.

The main goal of this work is to study geometric and analytic properties of nonnegative viscosity solutions of the following singular free boundary problem
\begin{equation}\label{P0}
\left\{
\begin{array}{ccl}
\Delta_\infty u = u^{-\gamma} & \mbox{in} & \Omega \cap \{u>0\} \\[0.15cm]
u = \varphi & \mbox{on} & \partial \Omega
\end{array}
\right.
\end{equation}
for parameters $0 \leq \gamma <1$. Here, $\Omega \subset \mathbb{R}^n$ is a bounded smooth domain and $\varphi \geq 0$ is a given smooth boundary data. Singular equations as in \eqref{P0} appear in several problems in engineering sciences and in contexts of simplified stationary models for fluids passing through a porous medium.
The Laplacian case, $\Delta u = u^{-\gamma}$, is fairly well understood. It appears as the Euler–Lagrange equation of the following non-differentiable functional
$$
\mathcal{J}_\gamma(u)=\int \frac{1}{2}|Du|^2 + u^{1-\gamma} \, dx .
$$
Regularity results for minimizers of $\mathcal{J}_\gamma$ has been studied in \cite{AP,GG1,GG2,P1,P2}. The nonvariational problem has been treated in \cite{AT} for a class of second order uniformly elliptic fully nonlinear operators. The pde satisfied in \eqref{P0} can be considered the intermediate case of the following free boundary problems: the infinity-obstacle problem, case $\gamma=0$ (see \cite{ALT,RTU} for variational and nonvariational approaches); the infinity-cavitation problem, case $\gamma=1$ (see \cite{ATU3,CF1,RST} and \cite{T2} for mixed singular structures).

\subsection*{Main ideas and results}
The study of this type of free boundary problem has significant difficulties: (i) the nonvariational sense, where no measure-distributional structure is available; (ii) the source term blows up along the a priori unknown set $\partial\{u>0\}$. In order to circumvent theses issues, we shall consider viscosity solutions of the penalized problem
\begin{equation}\label{Pe}\tag{$P_\varepsilon$}
\left\{
\begin{array}{ccl}
\Delta_\infty u = \mathcal{B}_{\varepsilon}(u)\, u^{-\gamma} & \mbox{in} & \Omega \\[0.15cm]
u = \varphi & \mbox{on} & \partial \Omega,
\end{array}
\right.
\end{equation}
where the term $\mathcal{B}_\varepsilon(s)$ is a suitable approximation of $\chi_{\{s>0\}}$.
Our main contribution is to provide uniform oscillation estimates for viscosity solutions of \eqref{Pe}, denoted by $u_\varepsilon$, obtaining so $C^{1,\alpha}$ estimates at free boundary points of limiting solutions
of \eqref{P0} (Fig. \ref{fig1}), for
\begin{equation}\label{alpha}
\alpha = \frac{4}{3+\gamma}.
\end{equation}

\smallskip
We now state our first main result.

\begin{theorem}[Optimal regularity at free boundary points]\label{mainthm}
Let $u$ be a limit solution of problem \eqref{P0}. For each subdomain $\Omega' \Subset \Omega$, there exist positive constants $C$ and $r_0$, depending only on $\gamma$, $\|u\|_{L^\infty(\Omega)}$, $\dist(\Omega',\partial\Omega)$ and dimension, such that, for points
$$
x \in \partial \{u >0\} \cap \, \Omega',
$$
there holds
\begin{equation}\label{optgrowth}
\sup_{B_r(x)}u \leq C\, r^\alpha
\end{equation}
for any $0< r \leq r_0$. Furthermore,
$$
\partial\{u>0\} \subset \{|Du| = 0\},
$$
which implies that $u$ is $C^{1,\frac{1-\gamma}{3+\gamma}}$ along $\partial\{u>0\}$.
\end{theorem}

By obtaining an entire radial supersolution for \eqref{Pe}, we provide nondegeneracy properties uniform on $\varepsilon$. Next, solutions for \eqref{P0} are limit of minimal Perron’s solutions of \eqref{Pe} (see Section \ref{singprob}).

\begin{theorem}[Nondegeneracy estimates]\label{nondeglim} Let $u$ be a limit Perron’s solution of problem \eqref{P0}. There exists universal $c>0$, depending only on $\gamma$, such that for
$$
x \in \overline{\{u>0\}} \cap \Omega,
$$
there holds
\begin{equation}\label{nondeg1}
\sup_{B_r(x)}u \geq c\, r^\alpha,
\end{equation}
for any $0< r \leq \frac{1}{2}\dist(x,\partial\Omega)$.
\end{theorem}

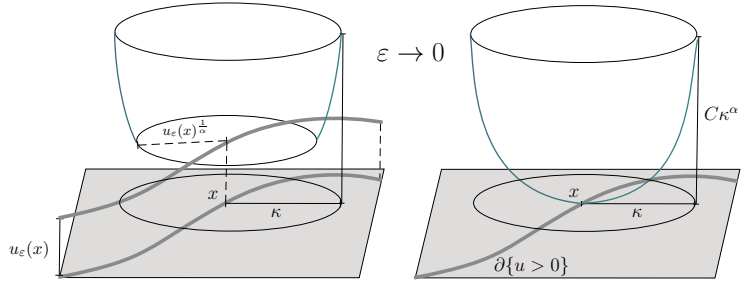
\begin{figure}[ht]
\centering
\psscalebox{0.6 0.6}
{
\begin{pspicture}(0,-2.2)(20.85587,4.658495)
\definecolor{colour0}{rgb}{0.8745098,0.8627451,0.8627451}
\definecolor{colour1}{rgb}{0.5411765,0.5411765,0.5411765}
\definecolor{colour2}{rgb}{0.23921569,0.44313726,0.5019608}
\definecolor{colour3}{rgb}{0.23137255,0.5058824,0.5411765}
\pspolygon[linecolor=black, linewidth=0.02, fillstyle=solid,fillcolor=colour0](10.69821,-1.4422625)(17.339062,-1.4459306)(17.891653,0.9644098)(11.280826,0.9644098)
\pspolygon[linecolor=black, linewidth=0.02, fillstyle=solid,fillcolor=colour0](2.8154514,-1.4353659)(9.456303,-1.4390341)(10.008895,0.9713064)(3.3980668,0.9713064)
\rput[bl](1.7,-1.0){\large{$u_\varepsilon(x)$}}
\rput[bl](5.1,1.65){\small{$u_\varepsilon(x)^{\frac{1}{\alpha}}$}}
\rput[bl](14.1,0.35){\Large{$x$}}
\psellipse[linecolor=black, linewidth=0.01, dimen=outer](6.5494037,4.0175676)(2.5089061,0.64092726)
\rput[bl](15.5,-0.1){\Large{$\kappa$}}
\psline[linecolor=black, linewidth=0.02, linestyle=dashed, dash=0.17638889cm 0.10583334cm](6.5244584,1.5521922)(6.497105,0.21901333)
\psline[linecolor=black, linewidth=0.02, linestyle=dashed, dash=0.17638889cm 0.10583334cm](9.928971,1.9870493)(9.913627,0.6990488)
\psbezier[linecolor=colour1, linewidth=0.08](2.7888398,-1.428636)(4.916617,-1.0142117)(4.7244763,-0.746617)(6.5124497,0.22906100576750305)(8.300423,1.2047391)(9.958078,0.7229817)(9.958078,0.7229817)
\psline[linecolor=black, linewidth=0.02, tbarsize=0.07055555cm 5.0]{|*-|*}(2.8257558,-0.13070036)(2.8104112,-1.3848171)
\psbezier[linecolor=colour2, linewidth=0.03](4.03857,3.9894524)(4.03857,3.045654)(4.2887907,1.9718623)(4.525546,1.5541070905946912)
\psellipse[linecolor=black, linewidth=0.01, dimen=outer](6.520831,1.6034484)(1.9993614,0.5603073)
\psellipse[linecolor=black, linewidth=0.01, dimen=outer](14.431942,3.9636796)(2.5090299,0.63599193)
\psbezier[linecolor=colour1, linewidth=0.08](10.671193,-1.4405869)(12.799075,-1.0293539)(12.606925,-0.76381975)(14.394986,0.2043453084973271)(16.183048,1.1725104)(17.840786,0.6944627)(17.840786,0.6944627)
\psline[linecolor=black, linewidth=0.01, tbarsize=0.07055555cm 5.0]{|*-|*}(14.379906,0.20285228)(16.88187,0.19973123)
\psbezier[linecolor=colour3, linewidth=0.03](16.95147,3.8411415)(16.741907,2.6340635)(16.820475,0.2848699)(14.380941,0.21138588971254307)
\psbezier[linecolor=colour2, linewidth=0.03](11.9176445,3.938383)(12.017237,3.208884)(11.870495,0.46837536)(14.336007,0.21265842181270272)
\psbezier[linecolor=colour1, linewidth=0.08](2.7948442,-0.11281727)(4.9226213,0.30681986)(4.9226213,0.7151624)(6.518454,1.5544366407022538)(8.114286,2.3937109)(9.976091,2.0049167)(9.976091,2.0049167)
\psbezier[linecolor=colour3, linewidth=0.03](9.056788,4.0157022)(9.056788,3.571071)(8.824964,2.0881622)(8.515565,1.622808830066017)
\psline[linecolor=black, linewidth=0.01, tbarsize=0.07055555cm 5.0]{|*-|*}(9.097235,3.9679344)(9.069053,0.21342728)
\rput[bl](6.1,0.3){\Large{$x$}}
\rput[bl](7.5,-0.1){\Large{$\kappa$}}
\rput[bl](17.1,2.0540297){\Large{$C\kappa^\alpha$}}
\psellipse[linecolor=black, linewidth=0.01, dimen=outer](14.445736,0.21885185)(2.460754,0.63599193)
\psline[linecolor=black, linewidth=0.01, tbarsize=0.07055555cm 5.0]{|*-|*}(16.979994,3.9058654)(16.924225,0.19273762)
\rput[bl](9.85,3.3719885){\huge{$\varepsilon \to 0$}}
\psline[linecolor=black, linewidth=0.02, dimen=outer, dash=0.17638889cm 0.10583334cm, tbarsize=0.07055555cm 5.0]{|*-|*}(6.497147,0.20974883)(9.040491,0.20662779)
\psline[linecolor=black, linewidth=0.02, linestyle=dashed, dash=0.17638889cm 0.10583334cm, tbarsize=0.07055555cm 5.0]{|*-|*}(4.55,1.5)(6.52,1.6)
\psellipse[linecolor=black, linewidth=0.01, dimen=outer](6.6043563,0.2257484)(2.460754,0.63599193)
\rput[bl](12.45,-1.4){\Large{$\partial\{u>0\}$}}
\end{pspicture}
}
\caption{The picture represents optimal growth estimates in the region $\{0<u_\varepsilon \lesssim \kappa^\alpha\} \cap B_\kappa(x)$ (Theorem \ref{mainthme}), showing how it is suitably constructed for obtaining optimal growth estimates at the free boundary $\partial\{u>0\}$.}
\label{fig1}
\end{figure}

Estimate \eqref{nondeg1} means that for small balls $B_r(x)$ centered at $\partial\{u>0\}$, minimal solutions do not grow slower than $|x-y|^{\alpha}$, $y \in \Omega$. As a result, density estimates and fine geometric-measure properties for the free boundary $\partial\{u>0\}$ are obtained, see Section \ref{geom}.

\smallskip
The paper is organized as follows: In Section \ref{singprob}, we provide existence of Perron’s solution for the problem \eqref{Pe}. In Section \ref{regsec}, we obtain local oscillation estimates for viscosity solutions of \eqref{Pe}. In Section \ref{nondegsec}, Non-degeneracy estimates are established for minimal solutions $u_\varepsilon$. Section \ref{limitsec} contains the proofs of our main results. In Section \ref{geom}, further analytic and geometric consequences are established. Section \ref{radial} we provide a radial example.

\subsection*{Notations}
Hereafter in this paper, $\Omega$ will be a bounded smooth domain in $\mathbb{R}^n$. $B_r(x) \subset \mathbb{R}^n$ denotes the open $n$-dimensional ball with radius $r>0$ centered at $x \in \mathbb{R}^n$, and $B_r:=B_r(0)$. For a point $x \in \mathbb{R}^n$ and $\mathcal{O} \subset \mathbb{R}^n$, we define $\dist(x,\partial\mathcal{O})$ to be the distance between $x$ and the boundary of $\mathcal{O}$, denoted by $\partial\mathcal{O}$. $\mathcal{L}^n(\mathcal{O})$ denotes the $n$-dimensional Lebesgue measure. For $x \in \Omega$ such that $u(x)<\iota$, we say $x \in \{u<\iota\}$. Similarly, we consider $\{u>\iota\}$, $\{u \leq \iota\}$ and $\{u \geq \iota\}$.

\section{Singular perturbation strategy and existence of minimal solutions} \label{singprob}

In this section, we introduce the singularly perturbed scheme adopted throughout this paper. For $\alpha$ defined by \eqref{alpha} and $\delta_\alpha$ (to be chosen in Section \ref{nondegsec}), we consider a Lipschitz function $\mathcal{B}(s) : \mathcal{\mathbb{R}} \to \mathcal{\mathbb{R}}$, satisfying
\begin{equation}\label{HB2}
\begin{array}{ccl}
0 \leq \mathcal{B}(s) \leq 1 & \mbox{for} & s \geq 0 \\[0.15cm]
\mathcal{B}(s) = 0 & \mbox{for} & s \leq \delta/2 \\[0.15cm]
\mathcal{B}(s) = 1 & \mbox{for} & s \geq \delta
\end{array}
\end{equation}
for $0<\delta=\delta_\alpha$. For each small parameter $\varepsilon>0$, we consider
$$
\mathcal{B}_{\varepsilon}(s):= \mathcal{B}\left(\frac{s}{\varepsilon^\alpha}\right),
$$
which is a suitable $\varepsilon$-approximation of $\chi_{\{s>0\}}$, such that
$$
\mathcal{B}_{\varepsilon}(s)\, s^{-\gamma} \equiv 0 \quad \mbox{for } s< \delta/2\varepsilon^\alpha.
$$
We highlight the following scaling invariance: if $v$ solves \eqref{Pe}, then for any $\iota>0$, the function
$$
v_\iota(x):=\frac{v(\iota x)}{\iota^\alpha} \quad \mbox{in } B_1
$$
solves
\begin{equation}\label{scaling}
\Delta_\infty v_\iota = \mathcal{B}_{\frac{\varepsilon}{\iota}}(v_\iota)\, v_\iota^{-\gamma} \quad \mbox{in} \; B_{\frac 1\iota}.
\end{equation}

\medskip

We shall use the following result from \cite[Lemma 2.5]{CEG}.

\begin{proposition}[Lipschitz regularity for infinity subharmonic functions]\label{regforsub}
Let $u$ be a viscosity solution of $-\Delta_\infty u \leq 0$ in $\Omega$. Then $u \in W^{1,\infty}_{loc}(\Omega)$. Furthermore,
$$
|Du(x)| \leq \max\limits_{z \in \partial B_r(x)} \frac{u(z)-u(x)}{r} \leq \frac{2\|u\|_{L^{\infty}(\Omega)}}{r},
$$
for each $x \in \Omega$ and $r< \dist(x,\partial\Omega)$.
\end{proposition}

\smallskip

\subsection*{Existence of minimal solutions for the \textbf{}problem \eqref{Pe}}
Here, we are interested in solutions of \eqref{P0} which are limits of minimal Perron’s solutions of \eqref{Pe}, as $\varepsilon \to 0$. Note that the lack of monotonicity in \eqref{Pe} on the variable $u$ does not allow us to make use of a direct application of the classical Perron method, and so, we derive existence of viscosity solutions of the problem \eqref{Pe}. We mention the following result.

\begin{theorem}[\cite{ART2}, Theorem 2.1]\label{ARTexistence}
Let $\mathcal{G}: \Omega \times [0,\infty) \to \mathbb{R}$ be a bounded function, uniformly Lipschitz in the interval $[0,\infty)$. Assume $\mathcal{F}: \Omega \times \mathbb{R}^n \times Sym(n) \to \mathbb{R}$ satisfies the monotonicity condition: for any $x \in \Omega$, $\xi \in \mathbb{R}^n$ and $N,M \in Sym(n)$, there holds
$$
\mathcal{F}(x,\xi,N) \leq \mathcal{F}(x,\xi,M) \quad \mbox{whenever } N \leq M.
$$
Assume a priori $C^{0,\alpha}$ estimates for viscosity solutions of $\mathcal{F}(x,Du,D^2 u)=f(x) \in L^\infty(\Omega)$ and that the problem
\begin{equation}\label{perroneq}
\begin{array}{ccl}
\mathcal{F}(x,Du,D^2 u) = \mathcal{G}(x,u) & \mbox{in} & \Omega, \\[0.15cm]
u = \varphi & \mbox{on} & \partial \Omega
\end{array}
\end{equation}
admits subsolution $\underline u$ and supersolution $\overline u$ with $\underline u=\overline u=\varphi \in W^{2,\infty}(\partial \Omega)$, then the function
$$
v(x):= \inf_{w \in \mathcal{S}} w(x)
$$
is a continuous viscosity solution of \eqref{perroneq}, where
$$
\mathcal{S}:=\{\omega \in C{(\overline \Omega)} \, | \, w \mbox{ is a viscosity supersolution to \eqref{perroneq} and } \underline u \leq w \leq \overline u \mbox{ in } \overline\Omega\}.
$$
\end{theorem}

Note that the infinity Laplacian operator is given by $\mathcal{F}(\xi,M)=\langle M \xi,\xi \rangle$, which is monotone. Also, by construction, for each $\varepsilon>0$ the function $\mathcal{G}(s)= \mathcal{B}_\varepsilon(s)s^{-\gamma}$ is bounded and uniformly Lipschitz in $[0,\infty)$. We recall that viscosity solutions of $\Delta_\infty u \in L^\infty$ are locally Lipschitz and that functions $\overline u$ and $\underline u$ satisfying
\begin{equation}\label{subsuper}
\left\{
\begin{array}{cll}
\Delta_\infty \underline u = \varepsilon^{-\alpha\gamma} & \mbox{in } \Omega \\
\underline u = \varphi & \mbox{in } \partial\Omega
\end{array}
\right.
\quad \mbox{and} \quad
\left\{
\begin{array}{cll}
\Delta_\infty \overline u = 0 & \mbox{in } \Omega \\
\overline u = \varphi & \mbox{in } \partial\Omega,
\end{array}
\right.
\end{equation}
are subsolution and supersolution for \eqref{Pe} respectively. Therefore, as an immediate consequence of Theorem \ref{ARTexistence} and Proposition \ref{regforsub}, we have the following result.

\begin{theorem}[Existence of minimal solutions]\label{Peexistence}
Let $\Omega\subset \R^n$ be a smooth domain and $\varphi$ be a $W^{2,\infty}(\partial\Omega)$ nonnegative boundary datum. Then, for each $\varepsilon>0$, the problem \eqref{Pe} has a minimal viscosity solution $u_\varepsilon\in C(\overline{\Omega})$. Moreover, $\{u_\varepsilon\}_{\varepsilon>0}$ is locally Lipschitz-equicontinuous, and satisfies $0 \leq u_\varepsilon \leq \|\varphi\|_{L^{\infty}}$.
\end{theorem}

\begin{remark}\label{rem1}
As stated above, we assure that $u_\varepsilon$ is nonnegative and globally bounded. Indeed, if $\mathcal{O}^-(u):=\{x \in \overline{\Omega} \, | \, u(x)<0\}$ is nonempty. Since $\varphi \geq 0$, we would have $\mathcal{O}^-(u) \Subset Int(\Omega)$ and so, $u_\varepsilon$ would be an infinity-harmonic function in $\mathcal{O}^-(u)$ satisfying $u_\varepsilon = 0$ on $\partial \mathcal{O}^-(u)$. By the classical comparison principle \cite{J,LW}, $u_\varepsilon = 0$ in $\mathcal{O}^-(u)$, which is a contradiction. Finally, using comparison principle and \eqref{subsuper}, we conclude that $u_\varepsilon \leq \|\varphi\|_{L^{\infty}}$.
\end{remark}

\begin{remark}\label{rem}
We finish this section by verifying that, for each $u_\varepsilon \geq 0$, either $u_\varepsilon >0$ or $u_\varepsilon \equiv 0$ in $\Omega$. In fact, if we assume that $u_\varepsilon(x_0)=0$ for some $x_0 \in \Omega$, then from \eqref{Pe}, $u_\varepsilon$ is infinity harmonic in $\mathcal{O}_\varepsilon :=\{ x \in \Omega \, | \, u_\varepsilon < \frac{\delta}{2} \varepsilon^\alpha\}$. Therefore, by the strong maximum principle, we get $u_\varepsilon \equiv 0$ in $\mathcal{O}_\varepsilon$, and so by continuity, $u_\varepsilon$ cannot attain any positive value in $\Omega$.
\end{remark}

\section{Optimal oscillation estimates at floating level sets}\label{regsec}

In this section, we consider positive viscosity solutions of the perturbed singular equation
\begin{equation}\label{Ee}\tag{$E_\varepsilon$}
\Delta_\infty u = \mathcal{B}_\varepsilon(u)\,u^{-\gamma} \quad \mbox{in } \; \Omega,
\end{equation}
for $B_\varepsilon$ as in Section \ref{singprob}. We establish optimal growth estimates  
$$
\sup\limits_{B_\kappa(x)} u_\varepsilon \lesssim \kappa^{\alpha} \quad \mbox{for} \quad
x \in \{0<u_\varepsilon \lesssim \kappa^\alpha\},
$$
see Theorem \ref{mainthme}. Our analysis follows ideas in \cite{AT}, which is based on regularity properties of the following auxiliary function
$$
v(x) := u^{\frac{1}{\alpha}}(x) \quad \mbox{for }\; x \in \Omega,
$$
for $\alpha$ as in \eqref{alpha}. We compute
$$
Dv= \frac{1}{\alpha} u^{\frac{1}{\alpha}-1}Du,
$$
$$
D^2v= \frac{1}{\alpha} \left(\frac{1}{\alpha}-1 \right) u^{\frac{1}{\alpha}-2}Du \otimes Du + \frac{1}{\alpha} u^{\frac{1}{\alpha}-1}D^2u.
$$
By using equation \eqref{Ee}, one has
$$
\Delta_\infty v = \frac{1}{\alpha^3}\left(\frac{1}{\alpha}-1 \right)u^{\frac{3}{\alpha}-4}|Du|^4+\frac{1}{\alpha^3}u^{\frac{3}{\alpha}-3}B_\varepsilon(u)u^{-\gamma}.
$$
Therefore, by writing the equation above only in terms of $v$, we get
\begin{equation}\label{auxeq}
\Delta_\infty v = \left( (1-\alpha)|Dv|^4 + \frac{1}{\alpha^3}f \right)v^{-1},
\end{equation}
where $f(x):=B_\varepsilon(v^{\alpha}(x))$, which is bounded and nonnegative. From Remark \ref{rem}, we note that $v>0$ in $\Omega$, which implies that, for each $\varepsilon>0$, the equation above is derived everywhere in $\Omega$, using the language of viscosity solutions.

\medskip

We now derive asymptotic growth estimates.

\begin{theorem}[Asymptotic growth estimates]\label{asympthm} Given $\Omega'\Subset \Omega$ and $\mu \in (0,1)$, there exist constants $C$ and $\kappa_0$ depending on $\mu$, $\gamma$, $\|u\|_{L^\infty(\Omega)}$, $\dist(\Omega',\partial \Omega)$ and dimension, but independent of $\varepsilon$, such that if $u$ is a positive viscosity solution of \eqref{Ee}, then
\begin{equation}\label{asympest}
\sup_{B_\kappa(x)} u \leq \left(C \kappa^\mu+ u(x)^{\frac{1}{\alpha}} \right)^\alpha
\end{equation}
for any $x \in \Omega'$ and $0< \kappa \leq \kappa_0$.
\end{theorem}

\begin{proof}
Here, we use techniques introduced in \cite{IL90}. Since the estimates are local, we may assume $\Omega=B_1$ and $\Omega'=B_{1/2}$. It is enough to show that $v=u^{\frac{1}{\alpha}}$ is locally $C^{0,\mu}$ for any $0<\mu<1$. For a given positive viscosity solution $v$ of \eqref{auxeq}, we claim that
\begin{equation}\label{IL}
\Theta(x,y):=v(x)-v(y)-L \omega(|x-y|)-\varrho(|x|^2+|y|^2) \leq 0,
\end{equation}
for universal large parameters $L,\varrho$ and $\omega(t)=t_+^\mu$. We now assume the term $f$ in \eqref{auxeq} is a positive bounded function in $\Omega$, and so, the constants $L$ and $\varrho$ shall be obtained depending only on $\mu$, $\gamma$, $\|u\|_{L^\infty(B_1)}$, $\sup_{B_1}f$ and dimension.
We argue by contradiction that the claim fails. If $(x_0,y_0)$ is a maximum point of $\Theta$ in $\overline{B_{1/2}} \times \overline{B_{1/2}}$, we assume
\begin{equation}\label{contradiction}
\Theta(x_0,y_0) > 0.
\end{equation}
Taking $\varrho:=8\|v\|_\infty$, we assure that $(x_0,y_0)$ is an interior maximum point, and $x_0 \neq y_0$. As an adaptation of Jensen-Ishii’s approximation lemma \cite{CIL}, we use \cite[Lemma 1]{ATU3} to guarantee the existence of limiting sub-jet and super-jet
$$
(\xi_x,M_x) \in \overline{J}^{2,+}_{B_{1/2}}u(x_0) \quad \mbox{and} \quad (\xi_y,M_x) \in \overline{J}^{2,-}_{B_{1/2}}u(y_0)
$$
such that
\begin{equation}\label{righthandside1}
\langle M_x\xi_x,\xi_x \rangle - \langle M_y\xi_y,\xi_y \rangle
\leq 4L \omega''(\rho) \left(L \omega'(\rho)+ \varrho \rho \right)^2 + 16 \varrho \left(L^2 \omega'(\rho)^2 + \varrho^2\right)
\end{equation}
for $\rho:=|x_0-y_0|$.
Next, using \eqref{auxeq}, we obtain
\begin{equation}\label{lefthandside2}
\begin{array}{rcr}
\mathcal{A}:=\langle M_x\xi_x,\xi_x \rangle - \langle M_y\xi_y,\xi_y \rangle & \geq & \displaystyle\left( (1-\alpha)|\xi_x|^4 + \frac{1}{\alpha^3}f(x_0) \right)v(x_0)^{-1}\, \\[0.35cm]
& & \displaystyle - \left( (1-\alpha)|\xi_y|^4 + \frac{1}{\alpha^3}f(y_0) \right)v(y_0)^{-1}.
\end{array}
\end{equation}
In addition, we easily notice that
$$
L \mu |x_0-y_0|^{\mu-1}-2\varrho|y_0| \geq L\mu-2\varrho \geq 0,
$$
where the last inequality is obtained for $L \gg 1$ universal. From this, we get
$$
|\xi_y|^4 \geq \left( L \mu |x_0-y_0|^{\mu-1}-2\varrho|y_0| \right)^4 \geq \left( L\mu - 2 \varrho \right)^4.
$$
Therefore,
\begin{equation}\label{lefthandside1}
(1-\alpha)|\xi_y|^4 + \frac{1}{\alpha^3}f(y_0) \leq (1-\alpha)\left( L\mu - 2 \varrho \right)^4 + \frac{1}{\alpha^3}\|f\|_{\infty} <0,
\end{equation}
provided $L$ is sufficiently large. By \eqref{contradiction},
$$
v(y_0)^{-1}>v(x_0)^{-1}
$$
and so, this together \eqref{lefthandside2} and \eqref{lefthandside1} provides
\begin{equation}\nonumber
\mathcal{A} \geq \displaystyle\left( (1-\alpha)|\xi_x|^4 - (1-\alpha)|\xi_y|^4 - \frac{1}{\alpha^3}\|f\|_\infty \right)v(x_0)^{-1},
\end{equation}
which gives
\begin{equation}\label{lefthandside3}
\mathcal{A}\, v(x_0) \geq (1-\alpha)(L\mu \rho ^{\mu-1} + 2 \varrho)^4 - (1-\alpha)(L\mu \rho^{\mu-1} - 2 \varrho)^4 - \frac{1}{\alpha^3}\|f\|_\infty .
\end{equation}
\medskip
On the other hand, from \eqref{righthandside1} we derive
\begin{equation}\label{righthandside2}
\begin{array}{rcl}
\mathcal{A} & \leq & 4\mu^3(\mu-1)L^3 \rho^{3\mu-4} +16\mu^2 \varrho L^2 \rho^{2\mu -2} + 16 \varrho^3 \\[0.3cm]
& \leq & \rho^{3\mu-4} \left( 4\mu^3(\mu-1)L^3 +16\mu^2 \varrho L^2 \rho^{2-\mu} + 16 \varrho^3 \rho^{4-3\mu} \right)\\[0.3cm]
& \leq & \rho^{3\mu-4} \left( 4\mu^3(\mu-1)L^3 +16\mu^2 \varrho L^2 + 16 \varrho^3 \right),
\end{array}
\end{equation}
where the last term turns strictly negative for $L$ chosen universally large.

\smallskip

Finally, from \eqref{contradiction}, we notice that
$$
v(x_0) \geq L \rho^\mu.
$$
From \eqref{lefthandside3}, \eqref{righthandside2} and the inequality above, we get
\begin{equation}\nonumber
\begin{array}{c}
(1-\alpha)(L\mu \rho ^{\mu-1} + 2 \varrho)^4 - (1-\alpha)(L\mu \rho^{\mu-1} - 2 \varrho)^4 - \dfrac{1}{\alpha^3}\|f\|_\infty \\[0.35cm]
\leq\; \rho^{4\mu-4} \left( 4\mu^3(\mu-1)L^4 +16\mu^2 \varrho L^3 + 16 \varrho^3L \right),
\end{array}
\end{equation}
which gives
\begin{equation}\nonumber
\begin{array}{c}
(1-\alpha)\left(\mu + \dfrac{2 \varrho}{L}\rho ^{1-\mu} \right)^4 - (1-\alpha)\left(\mu - \dfrac{2 \varrho}{L}\rho ^{1-\mu} \right)^4 - \dfrac{1}{\alpha^3L^4}\|f\|_\infty \rho ^{4-4\mu} \\[0.4cm]
\leq \; 4\mu^3(\mu-1) +16\dfrac{\mu^2 \varrho}{L}+ 16 \dfrac{\varrho^3}{L^3}.
\end{array}
\end{equation}
Therefore, we conclude that
\begin{equation}\nonumber
\begin{array}{c}
(1-\alpha)\left(\mu + \dfrac{2 \varrho}{L} \right)^4 - (1-\alpha)\left(\mu - \dfrac{2 \varrho}{L}\right)^4 - \dfrac{1}{\alpha^3L^4}\|f\|_\infty \\[0.4cm]
\leq \; 4\mu^3(\mu-1) +16\dfrac{\mu^2 \varrho}{L} + 16 \dfrac{\varrho^3}{L^3}.
\end{array}
\end{equation}
We get a contradiction by choosing $L$ larger than the previous choices.
\end{proof}

Next, we follow discrete iterative arguments and continuous methods as in \cite[theorem 3]{AT} to prove optimal growth estimates for positive solutions of \eqref{Ee}, independent of $\varepsilon$, see also \cite{CKS}.

\begin{theorem}[Asymptotic estimates imply optimality]\label{mainthme} Given $\Omega' \Subset \Omega$, there exist constants $C$ and $\kappa_\star$ depending on $\gamma$, $\|u\|_{L^\infty(\Omega)}$, $\dist(\Omega',\partial \Omega)$ and dimension, but independent of $\varepsilon$, such that if $u$ is a positive viscosity solution of \eqref{Ee}, then
\begin{equation}\label{growthest}
\sup_{B_\kappa(x)} u \leq C\left( \kappa^\alpha+ u(x) \right)
\end{equation}
for any $x \in \Omega'$ and $0< \kappa \leq \kappa_\star$.
\end{theorem}

\begin{proof}
With no loss of generality, we assume $\Omega=B_1$ and $\Omega'=B_{1/2}$. By contradiction, for each integer $k>1$, there exist $\varepsilon_k>0$ and viscosity solution $u_k$ of \eqref{Ee} for $\varepsilon=\varepsilon_k$, such that
\begin{equation}\label{contradiction2}
s_k:= \sup\limits_{B_{r_k}(x_k)}u_k \geq k(r_k^\alpha+u_k(x_k)),
\end{equation}
for some radii $r_k=o(1)$ and $x_k \in B_{1/2}$.

\smallskip
Define
$$
\varphi_k(x)=\frac{u_k(x_k+r_k x)}{s_k} \quad \mbox{in } \; B_1.
$$
Note that from \eqref{contradiction2} we obtain
\begin{equation}\label{function}
\sup\limits_{B_1}\varphi_k = 1 \quad \mbox{and} \quad \varphi_k(0) + \frac{r_k^\alpha}{s_k}\leq \frac{1}{k}.
\end{equation}
In addition, for each $k>0$, $\varphi_k$ solves
\begin{equation}\label{eqk}
\Delta_\infty \varphi_k = \mathcal{B}_{\epsilon_k}(\varphi_k) \varphi_k^{-\gamma} \quad \mbox{in } \; B_1,
\end{equation}
where
\begin{equation}\label{rhsk}
\mathcal{B}_{\epsilon_k}(s) := \left(\frac{r_k^\alpha}{s_k}\right)^{\frac{1}{3+\gamma}}B_{\epsilon_k}(s) \leq \left(\frac{r_k^\alpha}{s_k}\right)^{\frac{1}{3+\gamma}},
\end{equation}
for $\epsilon_k:=\varepsilon_k/r_k^{\frac{1}{\alpha}}$. Using the estimate above and \eqref{function}, we apply Theorem \ref{asympthm} and conclude that, for each $0<\mu<1$ fixed, $\varphi_k$ satisfies estimate \eqref{asympest}. Also, from \eqref{function}, Proposition \ref{regforsub} provides that $\{\varphi_k\}_k$ is equicontinuous in the $C^{0,1}(B_1)$-topology, and so, up to a subsequence, $\varphi_k$ converges to a function $\varphi_0$. By \eqref{eqk}, \eqref{rhsk} and stability of viscosity solutions, we conclude that $\varphi_0$ solves
$$
\varphi_0^\gamma\Delta_\infty \varphi_0 = 0 \quad \mbox{in } \; B_1,
$$
satisfies
$$
\varphi_0 \geq 0 \; \mbox{ in } B_1, \quad \sup\limits_{B_1}\varphi_0 =1, \quad \varphi_0(0)=0
$$
and, for each $0< \mu < 1$, there holds
\begin{equation}\label{asymplim}
\sup_{B_r(x)} \varphi_0 \leq \left(C r^\mu+ \varphi_0(x)^{\frac{1}{\alpha}} \right)^\alpha,
\end{equation}
for any $x \in \Omega'$ and $0< r \leq \kappa_0$.
\smallskip
Since $\partial\{\varphi_0>0\} \cap B_1 \neq \emptyset$ and $\{\varphi_0 >0\} \cap B_1 \neq \emptyset$, we can select a point $z_0 \in \{\varphi_0 =0\} \cap B_1$ and $z_+ \in \{\varphi_0>0\} \cap B_1$, satisfying
$$
d:=\dist(z_+, \{\varphi_0=0\})=|z_+-z_0|.
$$
Note that $\varphi_0$ is infinity-harmonic in $B_{d}(z_+)$. By the Hopf maximum principle for degenerate elliptic equations, see \cite{BdL}, we obtain
$$
0< \liminf\limits_{s \to 0^+} \frac{\varphi_0(z_0+s(z_+-z_0))-\varphi_0(z_0)}{s}.
$$
On the other hand, by \eqref{asymplim} and choosing $1/\alpha < \mu < 1$, we get
\begin{equation}\nonumber
\dfrac{\varphi_0(s(z_+-z_0)+z_0)}{s} = \dfrac{\varphi_0(s(z_+-z_0)+z_0)}{s^{\mu\alpha}} \cdot s^{\mu\alpha-1} \leq C s^{\mu\alpha-1}
\to 0,
\end{equation}
as $s \to 0^+$. This completes the proof of Theorem \ref{mainthme}.
\end{proof}

\section{Nondegeneracy properties for minimal solutions}\label{nondegsec}

Here, we prove nondegeneracy estimates for minimal solutions. More precisely, we show that the maximum a solution $u_\varepsilon$ in $B_r(x) \Subset \Omega$ does not grow slower than $r^{\alpha}$.

\begin{theorem}[Strong nondegeneracy]\label{nondegthm}
Let $u_\varepsilon$ be the minimal solution of \eqref{Pe}. There exists $c>0$ depending only on $\gamma$, such that for each $x\in\{u_\varepsilon \geq c \,\varepsilon^\alpha\}$ one has
\begin{equation}\label{nondegest}
\sup\limits_{B_r(x)}u_\varepsilon \geq c\, r^{\alpha},
\end{equation}
for $\varepsilon \leq r \leq \tfrac{1}{2}\dist(x,\partial\Omega)$.
\end{theorem}

The crucial step in obtaining nondegeneracy estimates is to construct proper supersolutions in $\mathbb{R}^n$.

\begin{proposition}\label{barrierprop1}
There exist positive constants $c$ and $\sigma$ depending only on $\gamma$, such that, for each $\eta\geq 1$ given, there exists a radially symmetric function $\Phi_\eta \in C^{1,1}(\mathbb{R}^n)$, satisfying
\begin{equation}\label{barrier1}
\Delta_\infty \Phi_\eta \leq\mathcal{B}(\Phi_\eta)\,\Phi_\eta^{-\gamma}
\end{equation}
pointwise in $\mathbb{R}^n$, where
\begin{equation}\label{barrier2}
\Phi_\eta\equiv c \; \mbox{ in } B_{\sigma\eta} \quad\quad \mbox{and} \quad\quad \Phi_\eta\geq c\, \eta^\alpha \; \mbox{ in } \mathbb{R}^n \setminus B_{\eta}.
\end{equation}
\end{proposition}

\begin{proof}
We define $\Phi_\eta$ by
$$
\Phi_\eta(x)=\left\{
\begin{array}{ccccc}
c & \mbox{for} & |x| \leq \sigma\eta \\[0.19cm]
A\left(|x|-\sigma\eta\right)^2 + c &\mbox{for} & \sigma\eta < |x|\leq \eta \\ [0.21cm]
2\delta|x|^{\alpha} + D & \mbox{for} &\eta < |x|,
\end{array}
\right.
$$
where $A, D, c$ and $\sigma$ are constants to be chosen later. Our first goal is to select parameters such that $\Phi_\eta \in C^{1}(\mathbb{R}^n)$. Note that it holds at points where $|x|=\sigma \eta$. For points $|x|=\eta$, we assume
$$
A(1-\sigma)^2\eta^2 + c= 2c\,\eta^\alpha +D,
$$
and so,
\begin{equation}\nonumber
A=\frac{1}{(1-\sigma)^2}\left[2c\,\eta^{\alpha-2} +\eta^{-2}(D-c)\right].
\end{equation}
In order to have continuity for $D\Phi_\eta$ along $|x|=\eta$, we take
\begin{equation}\nonumber
A=\frac{c\,\alpha}{1-\sigma}\eta^{\alpha-2}.
\end{equation}
Therefore, from the two last identities, one concludes
\begin{equation}\label{barriereq1}
\frac{c \,\alpha}{1-\sigma}\eta^{\alpha-2} = \frac{1}{(1-\sigma)^2}\left[2c\,\eta^{\alpha-2} +\eta^{-2}(D-c)\right].
\end{equation}
In particular, taking $D:=c(1-\eta^\alpha)$ in \eqref{barriereq1}, we only have to guarantee that
$$
\frac{c\,\alpha}{1-\sigma}\eta^{\alpha-2} = \frac{1}{(1-\sigma)^2}c\,\eta^{\alpha-2},
$$
which is true for
$$
\sigma := 1-\frac{1}{\alpha}=\frac{1-\gamma}{4}\in (0,\tfrac 14).
$$
\smallskip

Thus,
$$
\Phi_\eta(x)=\left\{
\begin{array}{ccccc}
c\, & \mbox{for} & 0 \leq |x| \leq \sigma\eta \\[0.2cm]
c\, \left[\alpha^2\eta^{\alpha-2}\left(|x|-\sigma\eta\right)^2 +1\right] &\mbox{for} & \sigma\eta\leq |x|\leq \eta \\ [0.2cm]
c\, \left[2\left(|x|^{\alpha} - \dfrac{\eta^2}{2} \right) + 1\right]& \mbox{for} &\eta \leq |x|.
\end{array}
\right.
$$
We then show that $\Phi_\eta$ satisfies \eqref{barrier1}.
\smallskip
Taking $c=\delta$, we easily check that
$$
\Delta_\infty \Phi_\eta(x)=0< c^{-\gamma}=\mathcal{B}(\Phi_\eta(x))(\Phi_\eta(x))^{-\gamma},
$$
holds for points $|x|\leq \sigma\eta$.
\smallskip
Next, for points $\sigma\eta\leq |x|\leq \eta$, we have
\begin{eqnarray*}
\Delta_\infty\Phi_\eta(x) & = & \Phi_\eta''(|x|)\Phi_\eta'(|x|)^2\\[0.1cm]
&=& 8c^3\alpha^6\eta^{3(\alpha-2)} \left(|x|-\sigma\eta\right)^2.
\end{eqnarray*}
Now, using the fact that $3\alpha-4=-\gamma \alpha$, we obtain
\begin{equation}\label{barriereq2}
\Delta_\infty\Phi_\eta(x) \leq 8c^3\alpha^4\eta^{-\alpha\gamma}.
\end{equation}
Since $\eta \geq 1$, we get
$$
c\leq \Phi_\eta(x)\leq c \left(\eta^\alpha +1\right) \leq 2c\, \eta^\alpha.
$$
In addition, we notice that $\mathcal{B} \equiv 1$ in $[\,\delta, \infty)$, and so
\begin{equation}\label{barriereq3}
\mathcal{B}(\Phi_\eta(x))\Phi_\eta(x)^{-\gamma} \geq (2c)^{-\gamma}\eta^{-\alpha\gamma}.
\end{equation}
Therefore, using \eqref{barriereq2} together with \eqref{barriereq3} and taking
$$
\delta=\delta_\gamma:=\frac{1}{2} \sqrt[3+\gamma]{\frac{1}{\alpha^4}},
$$
we conclude that \eqref{barrier1} holds for this region.
\smallskip
Finally, if $|x|\geq \eta$, direct computation gives
\begin{equation}\label{barriereq4}
\Delta_\infty \Phi_\eta(x)= 8c^3\alpha^3(\alpha-1)|x|^{-\alpha\gamma}.
\end{equation}
Also, since $\eta\geq 1$, we have
\begin{equation}\label{barriereq5}
\mathcal{B}(\Phi_\eta(x))\Phi_\eta(x)^{-\gamma} \geq c^{-\gamma}\left(2\left(|x|^\alpha-\frac{\eta^\alpha}{2}\right)+1\right)^{-\gamma} \geq (2c)^{-\gamma}|x|^{-\alpha\gamma}.
\end{equation}
Since
$$
c=\delta_\gamma \leq \frac{1}{2} \sqrt[3+\gamma]{\frac{1}{\alpha^3(\alpha-1)}},
$$
we use \eqref{barriereq4} and \eqref{barriereq5} to derive \eqref{barrier1}.
\smallskip
In conclusion, \eqref{barrier2} holds and the proof of Proposition \ref{barrierprop1} is complete.
\end{proof}

Next, for \eqref{Pe}, we prove existence of a radial supersolution $\Phi_\varepsilon$. That, combined with the minimality of $u_\varepsilon$, gives nondegeneracy.

\begin{proposition}\label{ebarrierprop}
For each $\varepsilon>0$ and $r\geq\varepsilon$, the radially symmetric function
$$
\Phi_\varepsilon(x):=\varepsilon^\alpha\Phi_{\frac r \varepsilon}\left(\frac{x}{\varepsilon}\right), \quad x \in \mathbb{R}^n
$$
is a supersolution of \eqref{Pe}. Moreover,
\begin{equation}\label{ebarrier}
\Phi_\varepsilon\equiv c\,\varepsilon^\alpha \quad \mbox{ in} \; B_{\sigma r} \quad \mbox{ and } \quad \Phi_\varepsilon\geq c\, r^\alpha \quad \mbox{in } \mathbb{R}^n\setminus B_{r}.
\end{equation}
\end{proposition}

\begin{proof}
Take $\eta= {r}/{\varepsilon}\geq 1$ in Proposition \ref{barrierprop1}. Arguing as in \eqref{scaling}, we conclude that $\Phi_\varepsilon$ is a supersolution of \eqref{Pe}. We also note that \eqref{ebarrier} follows directly from \eqref{barrier2}.
\end{proof}

Finally, we are ready to prove the main result of this section.

\begin{proof}[Proof of Theorem \ref{nondegthm}]
For simplicity, we prove \eqref{nondegest} for $0 \in \{u_\varepsilon > c\, \varepsilon^\alpha\}$. By continuity, we extend such estimate for $\{u_\varepsilon \geq c\, \varepsilon^\alpha\}$. Let $\varepsilon \leq r \leq \dist(0,\partial\Omega)$ and consider $\Phi_\varepsilon$ as in Proposition \ref{ebarrierprop}. The minimality of $u_\varepsilon$, implies that there is $\xi_r\in \partial{B_r}$ such that
\begin{equation}\label{minimclaim}
u_\varepsilon(\xi_r)\geq \Phi_\varepsilon(\xi_r).
\end{equation}
In fact, if $u_\varepsilon < \Phi_\varepsilon$ in $\partial B_r$, set
$$
\omega_\varepsilon :=
\left\{
\begin{array}{cl}
\min\left\{u_\varepsilon, \Phi_\varepsilon\right\} & \mbox{in} \; \overline{B_r}\\
u_\varepsilon & \mbox{in} \; \Omega \setminus \overline{B_r},
\end{array}
\right.
$$
Note that $\omega_\varepsilon$ is a supersolution of \eqref{Pe}, and $\omega_\varepsilon=\varphi$ on $\partial\Omega$. Since $\omega_\varepsilon= u_\varepsilon$ on $\partial B_r$, we have that $\omega_\varepsilon$ is continuous in $\overline \Omega$. Hence, as in Theorem \ref{ARTexistence}, we conclude that $\omega_\varepsilon \in \mathcal{S}$. On the other hand,
$$
u_\varepsilon (0) > c\, \varepsilon^{\alpha}=\Phi_\varepsilon (0) = \omega_\varepsilon (0),
$$
which contradicts the minimality of $u_\varepsilon$. Therefore, from \eqref{minimclaim}, we obtain
$$
\sup\limits_{B_r}u_\varepsilon \geq \sup\limits_{\partial B_r}u_\varepsilon \geq u_\varepsilon(\xi_r)\geq \Phi_\varepsilon(\xi_r) \geq c\, r^{\alpha}.
$$
\end{proof}

\medskip

\section{The limit free boundary problem}\label{limitsec}
In this section, we analyse the limit free boundary problem of\eqref{Pe}. From Theorem \ref{Peexistence}, the family $\{u_\varepsilon\}_{\varepsilon>0}$ is bounded in $L^\infty(\Omega)$ and, by Proposition \ref{regforsub}, is pre-compact in $C^{0,1}_{loc}(\Omega)$ topology. Therefore, up to a subsequence,
\begin{equation}\label{limit2}
u_{\varepsilon} \longrightarrow u \quad \mbox{ as }\; \varepsilon \to 0.
\end{equation}
The next result reveals that $u$ is a viscosity solution of problem \eqref{P0}.
\begin{proposition}
The limit function $u$ in \eqref{limit2} is a viscosity solution of
\begin{equation}\label{limitpde}
\Delta_\infty u = u^{-\gamma} \quad \mbox{in }\; \{u>0\}.
\end{equation}
\end{proposition}
\begin{proof}
Let $u(y)=:\iota>0$, $y\in \Omega$. By continuity, we can find a small radius $\varrho$ such that
$$
u\geq \frac{\iota}{2} \quad \mbox{in}\; B_\varrho(y).
$$
Since $u_\varepsilon\rightarrow u$ uniformly over compact sets, for any $\varepsilon$ small enough, one has
$$
u_\varepsilon \geq \frac{\iota}{4} >\delta\varepsilon^\alpha.
$$
Therefore,
$$
\Delta_\infty u_\varepsilon={u_\varepsilon}^{-\gamma} \quad\mbox{in}\quad B_{\frac{\varrho}{2}}(y).
$$
The stability of viscosity solutions under uniform limits implies that $u$ solves \eqref{limitpde} in the viscosity sense.
\end{proof}

\subsection*{Proofs of the main results}
As a consequence of the uniform estimates obtained in Theorem \ref{mainthme} and Theorem \ref{nondegthm}, we obtain optimal regularity estimates for limit solutions.

\begin{proof}[Proof of Theorem \ref{mainthm}]
By continuity, it is enough to verify that \eqref{optgrowth} holds for $x \in \{u< \theta r^\alpha\}$, for $\theta$ close to $1$. Since $u_\varepsilon \to u$ uniformly over compact sets, for radius $r>0$, we can find a parameter $\varepsilon_r>0$ and a subdomain $\Omega' \Subset \Omega$, with $x \in \Omega'$, such that
$$
\sup_{\Omega'}|u-u_\varepsilon| < (1-\theta)r^\alpha,
$$
for each $0<\varepsilon \leq \varepsilon_r$. Consequently,
$$
u_\varepsilon(x) \leq |u_\varepsilon(x)-u(x)| + u(x) < r^\alpha.
$$
Hence, thanks to Theorem \ref{mainthme}, for some universal small parameter $r_0$, we derive
$$
\sup_{B_r(x)}u_\varepsilon \leq Cr^\alpha
$$
for any $r \leq r_0$ such that $B_r(x) \subset \Omega'$. Therefore, if $x \in \{u< \theta r^\alpha\}$, then
$$
\sup_{B_r(x)}u \leq \sup_{\Omega'}|u-u_\varepsilon| + \sup_{B_r(x)}u_\varepsilon \leq ((1-\theta) + C) r^\alpha.
$$
Observe that for each $x \in \partial\{u>0\}$ and unit vectors $e_i \in \mathbb{R}^n$, we have
$$
0 \leq \frac{u(x+te_i)}{t} \leq \sup\limits_{z \in B_t(x)}\frac{u(z)}{t} \leq C\,t^{\alpha-1} \to 0 \quad \mbox{ as } \; t \to 0^+,
$$
In particular, this guarantees that, at free boundary points, $u$ is differentiable and satisfies $Du=0$. Then $C^{1,\alpha-1}$ regularity estimates easily follow from estimate \eqref{optgrowth}.
\end{proof}

Now, we prove nondegeneracy of limit solutions.

\begin{proof}[Proof of Theorem \ref{nondeglim}]
By continuity it suffices to verify that \eqref{nondeg1} holds for points $x \in \{u>0\}$. For a given $r < \tfrac{1}{2}\dist(x,\partial\Omega)$, note that for $\varepsilon \leq \varepsilon_r \leq r$ sufficiently small, we have $x \in \{u_\varepsilon>c\,\varepsilon^\alpha\}$. Therefore, from Theorem \ref{nondegthm}, we obtain
$$
\sup_{B_r(x)}u_\varepsilon\geq c\, r^\alpha,
$$
for any $ \varepsilon \leq r \leq \tfrac{1}{2}\dist(x,\partial\Omega)$. By uniform Lipschitz regularity estimates for $\{u_\varepsilon\}$, we get \eqref{nondeg1} by letting $\varepsilon\rightarrow 0$.
\end{proof}

\smallskip

\section{Geometric measure estimates for the free boundary}\label{geom}
An important consequence of the previous results is uniformly density of the positivity set $\{u>0\}$, as well as $(n-\epsilon)$-Hausdorff measure estimates of the free boundary $\partial\{u>0\}$, see \eqref{hauss}.

\begin{corollary}[Uniform density estimates]\label{densityP0}
Given a subdomain $\Omega'\Subset\Omega$, there exists a positive constant $c_\star$, depending only on universal parameters, such that for $x\in \Omega'\cap \{u>0\}$ and $0<\kappa < \dist(\Omega',\partial \Omega)$, there holds
$$
\frac{\mathcal{L}^n\left(B_\kappa(x)\cap\{u>0\}\right)}{\mathcal{L}^n\left(B_\kappa(x)\right)}\geq c_\star.
$$
\end{corollary}

\begin{proof}
Initially, for $x \in \{u > 0\} \cap \Omega'$, we consider parameter $0<\kappa \leq \dist(\Omega',\partial\Omega)$. Theorem \ref{nondeglim} guarantees that there exists $z_0 \in \overline{B_{\kappa}(x)}$ such that
\begin{equation}\label{dens2}
u_\varepsilon(z_0) \geq c \kappa^\alpha.
\end{equation}
Set
$$
d:=\dist(z_0, \partial\{u>0\}).
$$
We claim that there exists universal $\tau \in (0,1)$ such that $B_{\tau\kappa}(z_0) \Subset B_d(z_0)$. In fact, if $z' \in \partial\{u>0\}$ is such that $|z'-z_0|=d$, from \eqref{dens2} and Theorem \ref{mainthm}, one has
\begin{equation}\label{dens3}
c \kappa^\alpha \leq u(z_0) \leq \sup\limits_{B_{d}(z')} u \leq Cd^{\,\alpha}.
\end{equation}
Therefore, choosing any
$$
\tau < \left(\frac{c}{C} \right)^{\frac{1}{\alpha}},
$$
and using \eqref{dens3} we conclude the proof of the claim. Hence,
\begin{equation}\label{porous1}
B_{\tau\kappa}(z_0) \subset \{u> 0\}
\end{equation}
which implies,
\begin{equation}\label{porous2}
|B_\kappa(x) \cap \{u > 0\} | \geq |B_\kappa(x) \cap B_{\tau\kappa}(z_0)| \geq c_\star \kappa^n,
\end{equation}
for some $c_\star>0$ universal.
\end{proof}

Recalling the definition of porous set, see for example \cite{KR}, we notice - from the ideas in the proof of Theorem \ref{densityP0}, especially from \eqref{porous1} and \eqref{porous2} - that $\partial\{u>0\}$ is a porous set with porosity $\tau$, for each
$$
0<\tau < \min\left\{ \frac{1}{2},\left(\frac{c}{C}\right)^{\frac{1}{\alpha}} \right\}.
$$
In fact, by definition, a subset $E \subset \mathbb{R}^n$ is called $\tau$-\textit{porous}, if for each small radii $\rho>0$ and each point $x \in E$ there exists $y \in \mathbb{R}^n$ such that $B_{\tau\rho}(y) \subset B_{\rho}(x) \setminus E$. In addition, $\sigma$-porous set has Hausdorff dimension not exceeding $n-\epsilon \tau^n$, for $\epsilon$ depending only on dimension, see for instance \cite[Theorem 2.1]{KR}. Therefore, for any $\Omega' \Subset \Omega$, we have
\begin{equation}\label{hauss}
\mathcal{H}^{n-\epsilon \tau^n}(\partial\{u>0\} \cap \Omega')< \infty.
\end{equation}

\section{Optimality via a radial example}\label{radial}

Finally, we highlight the optimality of our results by exhibiting radial limit solutions of \eqref{P0}. For $C,R>0$ satisfying the compatibility condition
$$
R= \left(\frac{C}{C_\alpha} \right)^{\frac{1}{\alpha}}, \quad \mbox{where} \quad C_\alpha:=\left(\frac{1}{\alpha^3 (\alpha-1)}\right)^{\frac{1}{3+\gamma}}.
$$
We have that the function
$$
\omega(x)=C_\alpha |x|^\alpha
$$
is a limit solution of \eqref{P0} for $\Omega=B_R$ and $\varphi=C$ on $\partial\Omega$. Indeed, we look at radial solutions of \eqref{Pe} with constant boundary datum. Thus, we consider the following one-dimensional problem,
\begin{equation}\label{radialeq}
\begin{array}{rcc}
\omega''(\omega')^2 = \mathcal{B}_\varepsilon(\omega)\, \omega^{-\gamma} & \mbox{in} & (0,R) \\[0.1cm]
\omega \geq 0 & \mbox{in} & (0,R).
\end{array}
\end{equation}
For each $s \in (0,R)$, set
$$
\omega_\varepsilon(s) := C_\alpha\left( s+\varepsilon \right)^\alpha.
$$
Assuming $\delta$ as in Section \ref{singprob} and $\delta \leq C_\alpha$, one has
\begin{equation}\label{radialeq2}
\omega_\varepsilon''(\omega_\varepsilon')^2=\omega_\varepsilon^{-\gamma} = \mathcal{B}_\varepsilon(\omega)\, \omega^{-\gamma} \quad \mbox{in } \; (0,R).
\end{equation}
Hence, taking $\omega_\varepsilon(x):= \omega_\varepsilon(|x|)$ for $x \in B_R$, we conclude that $\omega_\varepsilon$ is the minimal Perron’s solution of \eqref{Pe} in $B_R$ satisfying $\omega_\varepsilon \equiv C_\alpha(R+\varepsilon)^\alpha$ on $\partial B_R$. Therefore, by letting $\varepsilon \to 0$, we get $\omega_\varepsilon \to \omega$.
In conclusion, we observe that $\omega$ is a limit solution of \eqref{P0} in $B_R$, satisfies $\inf_{B_R}\omega_\varepsilon=\omega_\varepsilon(0)=C_\alpha \varepsilon^\alpha$ and the approximating $\varepsilon$-level set coincides with the limit singular set, i.e.,
\begin{equation}\label{inf}
\partial \{ \omega_\varepsilon > C_\alpha \varepsilon^\alpha\}= \{0\}=\partial \{\omega>0\}.
\end{equation}

\medskip

\subsection*{Acknowledgements.}
The authors would like to thank the anonymous referee for the careful and detailed reading and for your valuable suggestions and corrections.

This paper is part of the second author's Ph.D. thesis. GSS acknowledges support from CAPES-Brazil and would like to thank the Department of Mathematics at Universidade Federal da Para\'iba for the pleasant and productive period during his Ph.D. program at that institution. DJA thanks the Abdus Salam International Centre for Theoretical Physics (ICTP) for the great hospitality during his research visits. DJA and GSS are partially supported by CNPq and grant 2019/0014 Para\'iba State Research Foundation (FAPESQ).

\bibliographystyle{amsplain, amsalpha}

\end{document}